\documentclass[12pt,a4paper,oneside]{amsart}

\usepackage{amsthm}
\usepackage{amsmath}
\usepackage{amssymb}
\usepackage{amscd}
\usepackage[utf8]{inputenc}
\usepackage{typearea}
\usepackage{eufrak}
\usepackage{yfonts}
\usepackage{textcomp}
\usepackage{mathrsfs}
\usepackage{hyperref}
\usepackage[draft]{fixme} 
\usepackage{pdfsync}
\usepackage[active]{srcltx}
\usepackage{xypic}
\usepackage{verbatim}

\renewcommand{\bar}[1]{\overline{#1}}

\newcommand{\ZZ}{\mathbb{Z}}
\newcommand{\NN}{\mathbb{N}}
\newcommand{\KK}{\mathbb{K}}

\newcommand{\G}{\mathcal{G}}
\newcommand{\BB}{\mathcal{B}}
\newcommand{\spann}{\textrm{span}}
\newcommand{\HG}{\mathcal{H}}

\newcommand{\TT}{\mathcal{T}}

\newcommand{\MM}{\mathcal{M}}
\newcommand{\UU}{\mathcal{U}}
\newcommand{\MT}{\mathbb{M}}

\newcommand{\RG}{\mathcal{R}}

\newcommand{\e}{\mathbf{e}}
\newcommand{\CC}{\mathbb{C}}

\newcommand{\X}{\mathbf{X}}

\newcommand{\Ide}{\mathsf{I}}

\newcommand{\coker}{\operatorname{coker}}

\newtheorem{lemma}{Lemma}[section]
\newtheorem{corollary}[lemma]{Corollary}
\newtheorem{theorem}[lemma]{Theorem}
\newtheorem{proposition}[lemma]{Proposition}

\theoremstyle{definition}
\newtheorem{definition}[lemma]{Definition}

\newtheorem{example}[lemma]{Example}

\newtheorem{remark}[lemma]{Remark}

\begin{document}

\title[The self-similar infinite dihedral group]{The homology of the  groupoid of the self-similar infinite dihedral group}

\author{Eduard Ortega}
\address{Department of Mathematical Sciences\\
NTNU\\
NO-7491 Trondheim\\
Norway } \email{Eduard.Ortega@ntnu.no}

\author{Alvaro Sanchez}
\address{Departament de Matemàtiques \\
Universitat Autònoma de Barcelona\\
08193 Bellaterra (Barcelona)\\
Spain } \email{alvarosanchez@mat.uab.cat}


\thanks{The second-named author was partially supported by Project MTM2017-83487-P from the Ministerio de Economía, Industria y Competitividad (Spain).}

\subjclass[2010]{22A22, 46L80, 19D55, 37B05, 46L05}

\keywords{Ample groupoid, homology of étale groupoids, self-similar group, HK conjecture, AH conjecture, Property TR}

\date{\today}

\begin{abstract}
We compute the $K$-theory of the $C^*$-algebra associated to the self-similar infinite dihedral group, and the homology of its associated étale groupoid. We see that the rational homology differs from the $K$-theory, strongly contradicting a conjecture posted by Matui. Moreover, we compute the abelianization of the topological full group of the groupoid associated to the self-similar infinite dihedral group.  
\end{abstract}

\maketitle

\section{Introduction}
\label{Intro}

In \cite{Mat4} Matui conjectured that  the homology groups of a minimal effective, ample  groupoid $\G$ totally captures the $K$-theory of their associated reduced groupoid $C^*$-algebra, and called it the HK conjecture, that is,
$$K_i(C^*_r(\G))\cong \sum_{k=0}^\infty H_{i+2k}(\G,\ZZ)\qquad \text{for }i=0,1\,.$$
This conjecture has been proved for important classes of groupoids, like the transformation groupoids of Cantor minimal systems, products of Cuntz-Krieger groupoids and many others  (see \cite{FKPS,Mat2,Mat4,Ort}). More generally, it was proved in \cite{PY} that the conjecture is true whenever $H_n(\G,\ZZ)=0$ for $n\geq 3$.  
However, the conjecture  was shown to be false by Scarparo in \cite{Sca} for a class of odometers of the infinite dihedral group. Then one can ask if a weakening of the conjecture could be true, the \emph{rational HK-conjecture}, that is  
$$K_i(C^*_r(\G))\otimes \mathbb{Q}\cong \sum_{k=0}^\infty H_{i+2k}(\G,\ZZ)\otimes \mathbb{Q}\qquad \text{for }i=0,1\,.$$
The rational HK-conjecture is satisfied by the transformations groupoids of $\ZZ^n$-minimal actions on the Cantor space \cite[Section 3]{Mat4}. 
Scarparo's counter-example is still valid for $i=0$ but not for $i=1$. In here we  present a complete counter-example for the rational HK-conjecture. Our example arises from a self-similar action of the infinite dihedral group  over a rooted tree and its associated groupoid, that not only encodes the global action of the group but also the local self-similarities of this action \cite[Section 5]{Nek}. This groupoid is going to be minimal and locally contracting, and its associated (reduced) groupoid $C^*$-algebra is going to be simple and  purely infinite. This contrasts with the counter-examples of Scarparo, whose associated $C^*$-algebras were AF-algebras.  
Therefore, there is not an obvious relationship between the example presented here and the one by Scarparo other than the infinite dihedral group, that is in the core of both constructions. However, in view of \cite{PY} one can suspect that the existence of high non-trivial homology groups for the infinite dihedral group, in contrast with the groups involved in the cases where the HK-conjecture is proved true, could be the cause of this bad homological behaviour.

Topological full groups associated to dynamical systems (and more generally to étale groupoids) are perhaps best known for being complete invariants for continuous orbit equivalence (and groupoid isomorphism), and also for diagonal preserving isomorphism of the associated $C^*$-algebra. Roughly speaking, the topological full group $[[\G]]$ of an étale groupoid $\G$ consists of all homeomorphisms of the unit space $\G^{(0)}$ which preserve the orbits of the dynamical system in a continuous manner.  Topological full groups also provide means of constructing new groups with interesting properties, like the Higman-Thompson group and its generalizations \cite{Mat4}, and  most notably by providing the first examples of finitely generated simple groups that are amenable (and infinite) \cite{JM}. In the works made by Matui \cite{Mat2,Mat3,Mat4} it was conjectured that the topologically full group of a minimal effective, ample groupoid together with the first two homology groups fit in the exact sequence 
$$H_0(\G)\otimes \ZZ/2\ZZ\longrightarrow [[\G]]_{ab}\longrightarrow H_1(\G)\longrightarrow 0\,,$$
where $[[\G]]_{ab}$ is the abelianization of $[[\G]]$.
This so called AH-conjecture has been verified in several cases \cite{Mat4,NO,Sca}, and so far no counter-example has been found.

The paper is structured as follows. In section 1 we define the $C^*$-algebra $C^*(\Gamma,\X)$ associated to a self-similar group $(\Gamma,\X)$ given in \cite[Definition 3.1]{Nek} and we introduce an equivalent construction via crossed products, that is more accessible  for the computation of the  $K$-theory. In Section 2 we introduce a realization of the infinite dihedral group as a self-similar group given in \cite{GNS}. With this in hand we compute the $K$-theory of its $C^*$-algebra taking advantage of the crossed product construction. In Section 3 we introduce the groupoid $\G_{(\Gamma,\X)}$ of the self-similar group $(\Gamma,\X)$, that is the groupoid of germs of all the local self-similarities. This groupoid first introduced in \cite{Nek}, got a more modern treatment in \cite{EP}. We prove that the groupoid associated to the infinite dihedral self-similar group is a locally compact, Hausdorff, minimal ample groupoid, and its reduced groupoid $C^*$-algebra $C_r^*(\G_{(\Gamma,\X)})$ is simple and purely infinite, and  isomorphic to $C^*(\Gamma,\X)$. We compute the low homology groups of this groupoid using similar techniques to the ones used in \cite{Ort}.  Finally in Section 4 we verify, using techniques developed in \cite{NO}, the AH-conjecture for the groupoid $\G_{(\Gamma,\X)}$, that is, there exists an exact sequence
$$H_0(\G_{(\Gamma,\X)})\otimes \ZZ/2\ZZ\longrightarrow [[\G_{(\Gamma,\X)}]]_{ab}\longrightarrow H_1(\G_{(\Gamma,\X)})\longrightarrow 0\,.$$
Then we can use the computations done in Section 3 to show that $[[\G_{(\Gamma,\X)}]]_{ab}\cong \ZZ/2\ZZ$

 \section{K-theory of the $C^*$-algebra associated to the infinite dihedral self-similar group}
In this section we are going to define the $C^*$-algebra associated to the self-similar action of the infinite
dihedral group \cite{Nek}, and compute its $K$-theory using an alternative construction as crossed product \cite[Section 3]{Nek}.

Let $\X$ be a finite  alphabet. Let $\X^n$ denote the words of length $n$. The empty word $\emptyset$ will be the only word of length $0$. Let $\X^*:=\bigcup_{n=0}^\infty\X^n$ be the set of finite words over the alphabet $\X$. Then $\X^*$ has a natural structure of rooted tree in which every word $v\in \X^*$ is connected by an edge to all the words of the form $vx$ for $x\in \X$. The root of $\X^*$ is the empty word $\emptyset$.  
A \emph{self-similar group} $(\Gamma,\X)$ is a faithful action of a group $\Gamma$  on the rooted tree $\X^*$ by automorphisms, such that for  every $\gamma\in \Gamma$ and every $x\in \X$, there exists $\eta\in \Gamma$ and $y\in\X$ such that 
$$\gamma\cdot (xw)=y(\eta\cdot w)\,,$$ 
for every $w\in \X^*$. We will denote $\eta=\gamma_{|x}$.
We write the above equality formally as 
$$\gamma\cdot x=y\cdot \gamma_{|x}\,.$$

\begin{definition}[{\cite[Definition 3.1]{Nek}}]
	The $C^*$-algebra associated to a self-similar group $(\Gamma,\X)$ is the universal $C^*$-algebra $C^*(\Gamma,\X)$ generated by a set of unitaries $\{U_\gamma \}_{\gamma\in \Gamma}$ and isometries $\{S_x\}_{x\in \X}$ satisfying the following relations:
	\begin{enumerate}
		\item  The map $\gamma \mapsto U_\gamma$ is a unitary representation,
		\item $\sum_{x\in \X}S_xS^*_x=1$ and $S^*_xS_x=1$ for every $x\in \X$.
		\item For all $\gamma\in \Gamma$ and $x\in\X$:
		$$U_\gamma S_x=S_yU_{\gamma_{|x}}\,,$$
		whenever $\gamma\cdot x=y\cdot \gamma_{|x}$.
		
	\end{enumerate}
\end{definition}

Given $n\in \NN$, let  $\X^n$ be set of words of length $n$, and  we will denote by  $\MT_{\X^n}$ the ring of $|\X|^n\times |\X|^n$ matrices over $\mathbb{C}$ indexed by $\X^n$, and given $x,y \in \X^n$ let $\e^{(|\X|^n)}_{x,y}$ be the corresponding matrix unit.

Now let $(\Gamma,\X)$ be a self-similar group, with $|\X|=d$.
Let $\CC[\Gamma]=\spann\{U_\gamma:\gamma\in \Gamma\}$ be the group algebra, and given $n\in\NN$ let us define the map 
$$\Phi_n:\CC[\Gamma]\otimes \MT_{\X^n} \to \CC[\Gamma]\otimes \MT_{\X^{n+1}}\,,$$  
$$U_\gamma\otimes \e^{(d^n)}_{v,w}\mapsto\sum_{x\in \X} U_{\gamma_{|x}}\otimes \e^{(d^{n+1})}_{vy,wx}$$ for $\gamma\in \Gamma$ and  for $v,w\in \X^n$ and $y\in \X$ such that $\gamma\cdot x=y\cdot \gamma_{|x}$. Observe that $\Phi_n$ extends to a unital $*$-homomorphism $\Phi_n:C^*(\Gamma)\otimes \MT_{\X^{n}}\to C^*(\Gamma)\otimes \MT_{\X^{n+1}}$. We call $\Phi_n$ the \emph{matrix recursion map at the $n$-level}.

Let $\Gamma$ be the infinite dihedral group, that is,  
$$\Gamma=\langle a,b: a^2=b^2=e\rangle\,.$$
The group $\Gamma$ is amenable, and $K_0(C^*(\Gamma))\cong \ZZ^3\cong \left\langle  \left[ 1\right] , \left[ \frac{1+U_a}{2}\right] , \left[ \frac{1+U_{b}}{2}\right] \right\rangle $, where $U:\Gamma\to U(C^*(\Gamma))$ is the canonical representation of $\Gamma$ in the unitaries of $C^*(\Gamma)$, and $K_1(C^*(\Gamma))=0$. To simplify notation we will write $1=U_e$ (the unit of $C^*(\Gamma)$).

The group $(\Gamma,\X)$ is a self-similar group (see \cite[Example 4]{GNS}) with $\X=\{0,1\}$ and relations
$$a\cdot 0=1\cdot e,\qquad a\cdot 1=0\cdot e,\qquad b\cdot 0=0\cdot a\qquad\text{and}\qquad b\cdot 1=1\cdot b\,.$$
Observe that $a,b\in\text{Aut}(\X^{\infty})$ have order $2$. The matrix recursion map, is given  by 
$$\Phi_0:C^*(\Gamma)\to C^*(\Gamma)\otimes \MT_2$$
$$U_a\mapsto 1\otimes \e^{(2)}_{0,1}+1\otimes \e^{(2)}_{1,0}$$
$$U_b\mapsto U_a\otimes \e^{(2)}_{0,0}+U_b\otimes\e^{(2)}_{1,1}\,,$$
and in general 
$$\Phi_n:C^*(\Gamma)\otimes \MT_{2^n}\to C^*(\Gamma)\otimes \MT_{2^{n+1}}$$
$$U_a\otimes \e^{(2^n)}_{v,w}\mapsto 1\otimes \e^{(2^{n+1})}_{v0,w1}+1\otimes \e^{(2^{n+1})}_{v1,w0}$$
$$U_b\otimes \e^{(2^n)}_{v,w}\mapsto U_a\otimes \e^{(2^{n+1})}_{v0,w0}+U_b\otimes\e^{(2^{n+1})}_{v1,w1}\,.$$
Let  us define
$$\MM_\Gamma:=\varinjlim (C^*(\Gamma)\otimes \MT_{2^n},\Phi_n)\,.$$

\begin{proposition}\label{K-theory-sub}
Let $(\Gamma, X)$ be the self-similar infinite dihedral group. Then there exists a group isomorphism 

\begin{center}

$\Psi_0: K_0 (\MM_\Gamma) \rightarrow \ZZ[1/2]\oplus \ZZ$

\end{center}
such that 
\begin{center}

$\Psi_0 ([\Phi_{n,\infty}(1\otimes \e_{v,v}^{(2^n)})])=(1/2^{n},0)$, 

$\Psi_0 ([\Phi_{n,\infty}((\frac{1+U_a}{2})\otimes\e_{v,v}^{(2^{n+1})} )])= (1/2^{n+1},0)$,

$\Psi_0 ([\Phi_{n,\infty}((\frac{1+U_b}{2})\otimes\e_{v,v}^{(2^{n+1})} )])= (-\sum_{k=1}^n1/2^{k+1},1)$,

\end{center}
for every $v\in \X^n$.
Moreover, $K_1(\MM_\Gamma)=0$.
\end{proposition}

\begin{proof}

The second statement is trivial, since $K_1$ is a continuous functor and $K_1(C^*(\Gamma)\otimes \MT_{2^n})=0$ for every $n$.\\
To compute $K_0 (\MM_\Gamma)$ we must inspect the induced homomorphism 
$$\Phi_0^*:\ZZ^3\cong K_0(C^*(\Gamma))\to K_0(C^*(\Gamma)\otimes \MT_2)\cong K_0(C^*(\Gamma)) \cong\ZZ^3\,,$$
where the isomorphism $K_0(C^*(\Gamma)\otimes \MT_2)\cong K_0(C^*(\Gamma))$ is the one induced by natural inclusion $C^*(\Gamma)\to C^*(\Gamma)\otimes \MT_2$ given by $x\mapsto x\otimes \e^{(2)}_{0,0}$.
Let $P:=\frac{1+U_a}{2}$ and $Q:=\frac{1+U_b}{2}$, then $\ZZ^3\cong K_0(C^*(\Gamma))=\langle[1],[P],[Q]\rangle$. First,  we have that 
$$\Phi^*_0([1])=\left[ \left( \begin{array}{cc} 
1 & 0 \\ 0 & 1 
\end{array}\right) \right]=2[1]\in K_0(C^*(\Gamma))\,,$$
and 
$$\Phi^*_0([P])=\left[ \left( \begin{array}{cc} 
\frac{1}{2} & \frac{1}{2} \\ \frac{1}{2} & \frac{1}{2} 
\end{array}\right) \right]=\left[ \left( \begin{array}{cc} 
1 & 0 \\ 0 & 0 
\end{array}\right) \right]=[1]\in K_0(C^*(\Gamma)) \,.$$
Finally,
$$\Phi^*_0([Q])=\left[ \left( \begin{array}{cc} 
\frac{1+U_a}{2} & 0 \\ 0 & \frac{1+U_b}{2} 
\end{array}\right) \right]=[P]+[Q]\in K_0(C^*(\Gamma))\,.$$

Now,  given $n\in \NN$ we define $1_n=\Phi_{n,\infty}(1\otimes \e^{(2^n)}_{0^n,0^n})$, $P_n:=\Phi_{n,\infty}(P\otimes \e^{(2^n)}_{0^n,0^n})$ and $Q_n:=\Phi_{n,\infty}(Q\otimes \e^{(2^n)}_{0^n,0^n})$, we have the following  relations in $K_0(\MM_\Gamma)$:
$$[1_n]=2[1_{n+1}]$$
$$[P_n]=[1_{n+1}]$$
$$[Q_n]=[Q_{n+1}]+[P_{n+1}]=[Q_{n+1}]+[1_{n+2}]\,.$$

Therefore, we have that $K_0(\MM_\Gamma)\cong \ZZ[1/2]\oplus \ZZ$ with the identifications $[1_n]\mapsto (1/2^n,0)$, $[P_n]\mapsto (1/2^{n+1},0)$ and $[Q_n]\mapsto (-\sum_{k=1}^n1/2^{k+1},1)$.
\end{proof}

Let $S:\MM_\Gamma\to \MM_\Gamma$ be the map defined by $\Phi_{n,\infty}(x\otimes \e^{(2^n)}_{v,w})\mapsto \Phi_{n+1,\infty}(x\otimes \e^{(2^{n+1})}_{0v,0w})$, and let us define  $$\BB_\Gamma:=\varinjlim(\MM_\Gamma,S)\cong \MM_\Gamma\otimes\KK\,,$$ where $\KK$ is the $C^*$-algebra of compact operators over an infinite dimensional separable Hilbert space. Therefore, we can identify the $K$-theory of $\MM_\Gamma$ with the one of $\BB_\Gamma$ via the inclusion map $\MM_\Gamma\to \BB_\Gamma$ given by $x\mapsto S_{0,\infty}(x)$.

Let $\hat{S}:\BB_\Gamma\to \BB_\Gamma$ be the automorphism given by $$S_{k,\infty}(\Phi_{n,\infty}(x\otimes \e^{(2^n)}_{v,w}))\mapsto S_{{k-1},\infty}(\Phi_{n,\infty}(x\otimes \e^{(2^n)}_{v,w}))=S_{{k},\infty}(\Phi_{n+1,\infty}(x\otimes \e^{(2^{n+1})}_{0v,0w}))\,.$$

Now consider the crossed product $C^*$-algebra $\BB_\Gamma \rtimes_{\hat{S}}\ZZ$. It was shown in \cite[Theorem 3.7]{Nek} that $\BB_\Gamma\rtimes_{\hat{S}}\ZZ$ is Morita equivalent to $C^*(\Gamma,\X)$.


\begin{proposition}\label{K-theory}

Let $(\Gamma,\X)$ be the self-similar infinite dihedral group. Then there exist group isomorphisms 

\begin{center}

$K_0(\BB_\Gamma\rtimes_{\hat{S}} \ZZ)=K_0(C^*(\Gamma, X))\cong\coker(1-\hat{S})=\ZZ$ and $K_1(\BB_\Gamma\rtimes_{\hat{S}} \ZZ)=K_1(C^*(\Gamma, X))\cong\ker (1-\hat{S})=\ZZ$

\end{center}

\end{proposition}

\begin{proof}
Using the Pimsner-Voiculescu six-term exact sequence, and since $K_1(\MM_\Gamma)=K_1(\BB_\Gamma)=0$, it follows that 
$$K_0(\BB_\Gamma\rtimes_{\hat{S}} \ZZ)\cong\coker(1-\hat{S}^*)\qquad\text{and}\qquad K_1(\BB_\Gamma\rtimes_{\hat{S}} \ZZ)\cong\ker (1-\hat{S}^*)$$ where $\hat{S}^*:K_0(\BB_\Gamma)\to K_0(\BB_\Gamma)$ is the map induced by the automorphism $\hat{S}:\BB_\Gamma\to \BB_\Gamma$ defined above. Then it is
 enough to check what is the image of the generators of $K_0(\BB_\Gamma)$ by the automorphism $\hat{S}^*$.  
Given $n\in \NN$ we define $1_n=\Phi_{n,\infty}(1\otimes \e^{(2^n)}_{0^n,0^n})$ and $Q_n:=\Phi_{n,\infty}(Q\otimes \e^{(2^n)}_{0^n,0^n})$. 
We have seen in the Proposition \ref{K-theory-sub} that $\{[1_n],[Q_n]\}_{n\in\NN}$ generates $K_0(\BB_\Gamma)$. Now we have that 
$$\hat{S}^*([1_n])=[\hat{S}(1_n)]=[1_{n+1}]\qquad \text{and}\qquad \hat{S}^*([Q_n])=[\hat{S}(Q_n)]=[Q_{n+1}]=[Q_n]-[1_{n+2}]\,,$$
for every $n\in \NN$.

Therefore using the identification of Proposition \ref{K-theory-sub}, we have that $1-\hat{S}^*:\ZZ[1/2]\oplus \ZZ \to \ZZ[1/2]\oplus \ZZ$ is given by $(1/2^n,0)\mapsto (1/2^{n+1},0)$ and $(0,1)\mapsto (1/2^2,0)$. Therefore, the kernel is generated by $(1/2,-1)$, so it is isomorphic to $\ZZ$, and since $(1-\hat{S}^*)(\ZZ[1/2]\oplus \ZZ)=(\ZZ[1/2],0)$ it follows that the cokernel is also isomorphic to $\ZZ$.
\end{proof}

\begin{remark}
The self-similar infinite dihedral group $(\Gamma,\X)$ is the Iterated Monodromy Group of the rational function $f(z)=z^2-2$ (see \cite[Example 11]{GNS}). In \cite[Theorem 4.8]{Nek} it is shown how to compute the $K$-theory of the $C^*$-algebra associated to the Iterated Monodromy Group of a post-critically finite hyperbolic rational function. 
\end{remark}

\section{Low dimension homology of the groupoid associated to the self-similar infinite dihedral group.}

A groupoid $\G$ is a small category of isomorphisms. Then the \emph{source} of $g\in \G$ is $g^{-1}g$ and the \emph{range} of $g$ is $gg^{-1}$. We denote by $\G^{(0)}:=\{gg^{-1}:g\in \G\}$ the \emph{set of units} of $\G$. A \emph{topological groupoid} is a groupoid with a topology that makes the groupoid operations (product, inversion, source and range) continuous. A groupoid $\G$ is  \emph{effective} if the interior of $\textrm{Iso}(\G):=\{g\in \G: r(g)=s(g)\}$ is equal to $\G^{(0)}$, and $\G$ is \emph{principal} if $\textrm{Iso}(\G)=\G^{(0)}$. A groupoid $\G$ is \emph{minimal} if $r(s^{-1}(x))$ is dense in $\G^{(0)}$ for every $x\in \G^{(0)}$.  A groupoid $\G$ is \emph{étale} if the source and range maps are local homeomorphisms. We say that an étale groupoid $\G$ is \emph{ample} if $\G^{(0)}$ is zero-dimensional, i.e. admits a basis of compact open sets. 
A \emph{bisection} of $\G$ is an open subsets $U$ of $\G$ such that the restriction of the source and range maps to $U$ are homeomorphism. Observe that an étale groupoid has always a basis consisting of bisections, and ample groupoids have a basis of compact bisections.

For a subset $A\subseteq \G^{(0)}$ we define $\G_A=\{g\in \G: s(g),r(g)\in A\}$, and this is a subgroupoid of $\G$ with unit space $A$. We say that $A\subseteq \G^{(0)}$ is \emph{$\G$-full} if $r(s^{-1}(x))\cap A\neq \emptyset$ for every $x\in \G^{(0)}$.
Two étale groupoids $\G$ and $\HG$ with totally disconnected unit spaces are called \emph{Kakutani equivalent} if there exist full clopen subsets $A$ and $B$ of $\G^{(0)}$ and $\HG^{(0)}$ respectively, such that $\G_A\cong \HG_B$ (\cite[Definition 4.1]{Mat2}).

\begin{example}[Transformation groupoid] Let $\Gamma$ be a discrete group with unit $e$, and let us consider an action of $\Gamma$ on a locally compact Hausdorff space $X$. Then we define the \emph{transformation groupoid} $X\rtimes \Gamma$ as the set $X\times \Gamma$ with the product topology, such that 
$$s(x,\gamma)=(x,e)\,,\qquad r(x,\gamma)=(\gamma\cdot x,e)\qquad\text{and}\qquad (\gamma_1\cdot x,\gamma_2)(x,\gamma_1)=(x,\gamma_2\gamma_1)\,.$$
If $X$ is a Cantor space, then $X\rtimes \Gamma$ is an ample groupoid.
\end{example}
\begin{example}[groupoid of germs]
Let $X$ be a locally compact Hausdorff space and $G$ be  an inverse semigroup of local homeomorphisms of $X$, i.e. homeomorphisms between open subsets of $X$. Then given a pair $(g,x)\in G\times X$,  its $G$-germ $[g,x]$ is an equivalence class where $(g,x)\sim (g',x')$ if and only if $x=x'$ and $g$ and $g'$ coincide in a neighborhood of $x$. Then the \emph{groupoid of $G$-germs} $\G$ have operations given by
$$[g_1,x_1][g_2,x_2]=[g_1g_2,x_2]$$
if and only if $g_2(x_2)=x_1$, and 
$$[g,x]^{-1}=[g^{-1},g(x)]\,.$$
The topology of the groupoid of $G$-germs is given by the basis of open sets
$$\UU_{U,g}:=\{[g,x]:x\in U\}\,,$$
for $g\in G$ and $U$ an open subset of the domain of $g$. 

It is not hard to see that the groupoid of $G$-germs is always a locally compact, effective étale groupoid, and the above defined $\UU_{U,g}$ are bisections.         

\end{example}

Let $\pi:X\to Y$ be a local homeomorphism  between two locally compact Hausdorff spaces. Then given any $f\in C_c(X,\ZZ)$ we define 
$$\pi_*(f)(y):=\sum_{\pi(x)=y}f(x)\,.$$
It is not hard to show that $\pi_*(f)\in C_c(Y,\ZZ)$. 

Given an étale groupoid $\G$ and $n\in\NN$ we write $\G^{(n)}$ for the space of composable strings of $n$ elements in $\G$ with the product topology. For $i=0,\ldots,n$, we let $d_i:\G^{(n)}\to\G^{(n-1)}$ be the map defined by 
$$d_i(g_1,g_2,\ldots,g_n)=\left\lbrace \begin{array}{ll}
(g_2,g_3,\ldots,g_n) & \text{if }i=0\,, \\
(g_1,\ldots, g_{i}g_{i+1},\ldots, g_n) & \text{if }1\leq i\leq n-1\,, \\
(g_1,g_2,\ldots,g_{n-1})& \text{if }i=n\,.
\end{array}\right. $$

We define the homomorphism $\delta_n:C_c(\G^{(n)},\ZZ)\to C_c(\G^{(n-1)},\ZZ)$ given by
$$\delta_1=s_*-r_*\qquad\text{and}\qquad\delta_n=\sum_{i=0}^n(-1)^nd_{i*}\,.$$
Then we define the homology $H_*(\G)$ as the homology groups of the chain complex $C_\bullet(\G,\ZZ)$ given by
$$0\longleftarrow C_c(\G^{(0)},\ZZ)\longleftarrow^{\delta_1} C_c(\G^{(1)},\ZZ)\longleftarrow^{\delta_2} C_c(\G^{(2)},\ZZ)\longleftarrow\cdots$$

 It was proved in \cite[Theorem 3.6(2)]{Mat2} that two Kakutani equivalent groupoids have isomorphic homology groups.

Let $\Gamma$ be a countable discrete group and $\G$ an étale groupoid. When $\rho:\G\to\Gamma$ is a groupoid homomorphism, the \emph{skew product} $\G\times_\rho\Gamma$ is $\G\times\Gamma$ with the following groupoid structure: $(g,\gamma)$ and $(g',\gamma')$ are composable if and only if $g$ and  $g'$ are composable and $\gamma\rho(g)=\gamma'$, and then
$$(g,\gamma)\cdot(g',\gamma\rho(g))=(gg',\gamma)\qquad\text{and}\qquad(g,\gamma)^{-1}=(g^{-1},\gamma\rho(g))\,.$$
Given $n\in\NN$ we can define the action $\hat{\rho}:\Gamma\curvearrowright(\G\times_\rho \Gamma)^{(n)}$ by 
$$\hat{\rho}^\gamma((g_1,\gamma_1),\ldots,(g_n,\gamma_n))=((g_1,\gamma\gamma_1),\ldots,(g_n,\gamma\gamma_n))\,.$$

Now let $(\Gamma,\X)$ be a self-similar group.  
Let $\X^\infty$ be the set of infinite words $x_1x_2x_3\cdots $ in the alphabet $\X$, with the topology given by the sets
$$Z(\alpha)=\{\alpha x: x\in \X^\infty \}\,,$$
for $\alpha\in \X^*$. With this topology $\X^\infty$ is homeomorphic to the Cantor space if $|\X|\geq 2$. 
Then we have that $\Gamma$ acts by homeomorphism on $\X^\infty$ as follows: Given $g\in\Gamma$ we define $g\cdot x_1x_2x_3\cdots =y_1y_2y_3\cdots $ with  $g_{n}\cdot x_n= y_n \cdot g_{n+1}$ where $g_1=g$ and $g_{n+1}=(g_n)_{|x_n}$ for all $n$.
  
Let $\G_{(\Gamma,\X)}$ be the groupoid of $G$-germs of the local homeomorphisms $G=\{S_{(\alpha,g,\beta)}:\alpha,\beta\in \X^*\text{ and }g\in\Gamma\}$ of $\X^\infty$, where $S_{(\alpha,g,\beta)}:Z(\beta)\to Z(\alpha)$ by $\beta x\mapsto \alpha(g\cdot x)$ for every $x\in \X^\infty$, that is
$$\G_{(\Gamma,\X)}=\{ [\alpha,g,\beta; x]: \alpha,\beta\in \X^*,\,g\in \Gamma,\,x\in Z(\beta) \}\,,$$
where $[\alpha,g,\beta; x]=[\alpha',g',\beta'; x']$ if and only if $x=x'$ and there exists a neighborhood $U$ of $x$ such that $S_{(\alpha,g,\beta)}(y)=S_{(\alpha',g',\beta')}(y)$ of every $y\in U$. Observe then that $[\alpha,g,\beta; \beta \gamma x ]=[\alpha(g\cdot \gamma),g_{|\gamma},\beta\gamma; \beta \gamma x ]$. Then the groupoid operations are defined as follows: 
$$[S_{(\alpha',g',\beta')}, x']\cdot [S_{(\alpha,g,\beta)}, x]=[S_{(\alpha',g'g,\beta)}, \beta x]$$
if and only if $\beta'=\alpha$ and $x'=\alpha (g\cdot x)$, and 
$$[\alpha,g,\beta;\beta x]^{-1}=[\beta,g^{-1},\alpha; \alpha (g\cdot x)]\,.$$
Then we have that $$\G_{(\Gamma,\X)}^{(0)}=\{[v,e,v;x]:x\in\X^\infty \}\,,$$ so we can identify $\G_{(\Gamma,\X)}^{(0)}$ with $\X^\infty$.
The  open subsets of $\G_{(\Gamma,\X)}$ are 
$$Z(\alpha,g,\beta;U)=\{[\alpha,g,\beta; x]: x\in U  \}$$
for $\alpha,\beta\in \X^*$, $g\in \Gamma$ and an open subset $U$ of $Z(\beta)$.
With this topology $\G_{(\Gamma,\X)}$ is a  minimal, locally compact, effective and locally contracting ample groupoid \cite[Section 17]{EP}. 

\begin{definition}
	We will say that $(\Gamma,\X)$ is \emph{pseudo-free} if, whenever $g\cdot x =x\cdot e$ for some $g\in \Gamma$ and $x\in \X$, then $g=e$.
\end{definition}

\begin{lemma}
The self-similar infinite dihedral group $(\Gamma,\X)$ is pseudo-free. 
\end{lemma}
\begin{proof}
Let $g\in \Gamma$, then $g$ is of the form either $(ab)^na$, $(ab)^n$, $(ba)^nb$ or $(ba)^n$ for some $n\geq 0$. First we check that 
$$(ab)\cdot 0= 1\cdot a\qquad \text{and}\qquad \qquad (ab)\cdot 1=0\cdot b\,,$$
and then 
$$(ab)^{2n}\cdot 0=0\cdot (ba)^n\qquad \text{and}\qquad (ab)^{2n}\cdot 1=1\cdot (ab)^n\,,$$
meanwhile
$$(ab)^{2n+1}\cdot 0=1\cdot a(ba)^n\qquad \text{and}\qquad (ab)^{2n+1}\cdot 1=0\cdot b(ab)^n\,.$$
But then 
$$(ab)^{2n}a\cdot 0=1\cdot (ab)^n\,,\qquad (ab)^{2n}a\cdot 1=0\cdot (ba)^n\,,$$
$$(ab)^{2n+1}a\cdot 0=0\cdot b(ab)^n\qquad\text{and}\qquad (ab)^{2n}a\cdot 1=1\cdot a(ba)^n\,.$$
The above computations show that whenever $g$ is of the form $(ab)^n$ or $(ab)^na$, $g\cdot \alpha=\alpha\cdot e$ implies $g=e$.

Similar computations follows whenever $g$ is of the form $(ba)^nb$ or $(ba)^n$. 
\end{proof}

Since  $(\Gamma,\X)$ is a pseudo-free self-similar group it follows that 
$$[\alpha,g,\beta; x]=[\alpha,g',\beta; x]$$
if and only if $g=g'$ \cite[Proposition 8.6]{EP}. Moreover, the groupoid $\G_{(\Gamma,\X)}$ is Hausdorff \cite[Proposition 12.1]{EP}.

\begin{lemma}\label{minimal_action}
Consider the  self-similar infinite dihedral group $(\Gamma,\X)$. Then given any $\alpha,\beta\in \X^n$ for $n\in \NN$, there exists $g\in \Gamma$ such that $g\cdot \alpha=\beta$.
\end{lemma}
\begin{proof}
It is enough to prove that given $x,y\in \X=\{0,1\}$ there exist $g_{x,y},h_{x,y}\in \Gamma$ such that $$g_{x,y}\cdot x=y\cdot a\qquad\text{and}\qquad h_{x,y}\cdot x=y\cdot b\,.$$
Therefore one can check that 
$$g_{0,0}=b\,,\qquad h_{0,0}=aba\,,\qquad g_{1,0}=ba\,,\qquad h_{1,0}=ab\,,$$
$$g_{0,1}=ab\,,\qquad h_{0,1}=ba\,,\qquad g_{1,1}=aba\,,\qquad h_{1,1}=b\,.$$

\end{proof}

If $(\Gamma,\X)$ is the  self-similar group, then one can define a cocycle  $c:\G_{(\Gamma,\X)}\to \ZZ$ given by $c([S_{(\alpha,g,\beta)},  x])=|\alpha|-|\beta|$; then we have that the subgroupoid 
$$\HG_{(\Gamma,\X)}:=c^{-1}(0)=\{[\alpha,g,\beta;\beta x]:|\alpha|=|\beta|\}$$ is Kakutani equivalent to the skew product groupoid $\G_{(\Gamma,\X)}\times_c\ZZ$. We can decompose $\HG_{(\Gamma,\X)}$ as the union $\bigcup_{n\in \NN} (\HG_{(\Gamma,\X)})_n$ where 
$$(\HG_{(\Gamma,\X)})_n:=\{[\alpha,g,\beta;\beta x]:|\alpha|=|\beta|=n\}\,.$$ 
Observe that $(\HG_{(\Gamma,\X)})_n\subseteq (\HG_{(\Gamma,\X)})_{n+1}$,
and hence $H_*(\HG_{(\Gamma,\X)})\cong \varinjlim (H_*((\HG_{(\Gamma,\X)})_n),(\iota_n)_*)$, where $(\iota_n)_*$ is the map induced by the inclusion $\iota_n: (\HG_{(\Gamma,\X)})_n\to (\HG_{(\Gamma,\X)})_{n+1}$ \cite[Lemma 1.4]{Ort}.

\begin{proposition}\label{lower_hom}

Let $(\Gamma,X)$ be the self-similar infinite dihedral group, $\G_{(\Gamma,\X)}$ its associated groupoid of germs, and $\HG_{(\Gamma,\X)}:=c^{-1}(0)$. Then we have that  $H_0(\HG_{(\Gamma,\X)})\cong \ZZ[1/2]$ and $\HG_1(\HG_{(\Gamma,\X)})\cong\ZZ/2\ZZ$. 
\end{proposition}
\begin{proof}

To compute the homology groups, we must investigate the induced maps in homology of the natural inclusions $\iota_{n,m}:(\HG_{(\Gamma,\X)})_n\to (\HG_{(\Gamma,\X)})_{m}$ for $n<m$ ($\iota_n:=\iota_{n,n+1}$), that induce a map 
$$\iota_{n,m}:C_c((\HG_{(\Gamma,\X)})_n,\ZZ)\to C_c((\HG_{(\Gamma,\X)})_{m},\ZZ)$$
defined by  $1_{Z(\alpha,g,\beta;Z(\beta\gamma))}\mapsto 1_{Z(\alpha(g\cdot \gamma),g_{|\gamma},\beta\gamma;Z(\beta\gamma))}$ for $\alpha,\beta\in \X^n$, $\gamma\in \X^{m-n}$ and $g\in \Gamma$.
Identifying $(\HG_{(\Gamma,\X)})_n^{(0)}$ with $\X^\infty$, we claim  that $[1_{Z(\alpha)}]=[1_{Z(\beta)}]\in H_0((\HG_{(\Gamma,\X)})_n)$ if and only if $|\alpha|=|\beta|$. Indeed, we can assume that $|\alpha|,|\beta|\geq n$. Clearly if $[1_{Z(\alpha)}]=[1_{Z(\beta)}]$ then $|\alpha|=|\beta|$. Now, assume that $|\alpha|=|\beta|$, then by Lemma \ref{minimal_action} there exists $g\in \Gamma$ such that $g\cdot \alpha=\beta$, and hence $g\cdot Z(\alpha)=Z(\beta)$. Therefore, $[1_{Z(\alpha)}]=[1_{Z(\beta)}]\in H_0((\HG_{(\Gamma,\X)})_n)$, as desired.  Thus, we can identify $H_0((\HG_{(\Gamma,\X)})_n)$ with $\ZZ[1/2]$ sending $[1_{Z(\gamma)}]$ to $\frac{1}{2^{k}}$, where $\gamma\in \X^k$. Observe that this identification is independent of $n$. On the other hand, we have that $$(\iota_{n,m})_*:H_0((\HG_{(\Gamma,\X)})_n)\to H_0((\HG_{(\Gamma,\X)})_{m})$$
is given by $[1_{Z(\alpha)}]\mapsto [1_{Z(\alpha)}]$, for $\alpha\in\X^*$. Then using the above identifications of $H_0((\HG_{(\Gamma,\X)})_n)$ and $H_0((\HG_{(\Gamma,\X)})_m)$ with $\ZZ[1/2]$, the map $(\iota_{n,m})_*$ is the identity. Therefore we have that $ H_0(\HG_{(\Gamma,\X)})\cong \ZZ[1/2]$. 


On the other hand, we have that
\begin{align*}
H_1((\HG_{(\Gamma,\X)})_n) & =\ker \delta_1/\text{im }\delta_2 =\{[f]:  f\in C_c((\HG_{(\Gamma,\X)})_n, \ZZ): \delta_1(f)=0\}\,,
\end{align*}
Let $f\in C_c((\HG_{(\Gamma,\X)})_n, \ZZ)$ with $\delta_1(f)=0$. We can write 
$$f=\sum_{i=0}^k \lambda_i 1_{Z(\alpha_i,g_i,\beta_i;U_i)}\,,$$
where $\alpha_i,\beta_i\in \X^n$, $\lambda_i\in \ZZ$, $g_i\in \Gamma$ and $U_i$ are clopen subsets of $Z(\beta_i)$.
Replacing $[f]$ with $[\iota_{n,m}(f)]\in H_1((\HG_{(\Gamma,\X)})_m)$ for big enough $m$, we can assume that 
$$f=\sum_{i=0}^k \lambda_i 1_{Z(\alpha_i,g_i,\beta_i;Z(\beta_i))}\,,$$
where $\alpha_i,\beta_i\in \X^n$, $\lambda_i\in \ZZ$ and $g_i\in \Gamma$,

We have that 
$$\delta_2\left( 1_{(Z(\alpha,g,\beta;Z(\beta))\times Z(\beta,h,\gamma;Z(\gamma)))\cap  (\HG_{(\Gamma,\X)})_n^{(2)} }\right) =1_{Z(\alpha,g,\beta;Z(\beta))}+1_{Z(\beta,h,\gamma;Z(\gamma))}-1_{Z(\alpha,gh,\gamma;Z(\gamma))}\,,$$
for every $\alpha,\beta,\gamma\in  \X^n$ and $g,h\in \Gamma$, and therefore 
\begin{equation}\label{eq_0}\left[ 1_{Z(\alpha,g,\beta;Z(\beta))}\right] +\left[ 1_{Z(\beta,h,\gamma;Z(\gamma))}\right] =\left[ 1_{Z(\alpha,gh,\gamma;Z(\gamma))}\right] \end{equation}
in $H_1((\HG_{(\Gamma,\X)})_n)$.  Now if we choose $\beta=0^n$ and $h=e$ in (\ref{eq_0}), we have that 
\begin{equation}\label{eq_1}\left[ 1_{Z(\alpha,g,\gamma;Z(\gamma))}\right]=\left[ 1_{Z(\alpha,g,0^n;Z(0^n))}\right] +\left[ 1_{Z(0^n,e,\gamma;Z(\gamma))}\right] \end{equation}
in $H_1((\HG_{(\Gamma,\X)})_n)$, and applying the above relation again we have that 
$$\left[ 1_{Z(\alpha,g,\gamma;Z(\gamma))}\right]=\left[ 1_{Z(\alpha,e,0^n;Z(0^n))}\right]+\left[ 1_{Z(0^n,g,0^n;Z(0^n))}\right] +\left[ 1_{Z(0^n,e,\gamma;Z(\gamma))}\right]\,.$$

Moroever if $\alpha=\beta=\gamma$ and $g=h=e$, we have that 
$$\left[ 1_{Z(\alpha,e,\alpha;Z(\alpha))}\right]=\left[ 1_{Z(\alpha,e,\alpha;Z(\alpha))}\right]+\left[ 1_{Z(\alpha,e,\alpha;Z(\alpha))}\right]$$
in $H_1((\HG_{(\Gamma,\X)})_n)$, and hence $\left[ 1_{Z(\alpha,e,\alpha;Z(\alpha))}\right]=0$. In particular, combining the previous equality and equation (\ref{eq_1}) we have that 
$$\left[ 1_{Z(\alpha,e,0^n;Z(0^n))}\right]+\left[ 1_{Z(0^n,e,\alpha;Z(\alpha))}\right]=0\,,$$
in $H_1((\HG_{(\Gamma,\X)})_n)$.

Therefore, combining the above computations we can assume that  
$$f=\sum_{i=0}^k \lambda_i 1_{Z(0^n,g_i,0^n;Z(0^n))} + \sum_{\alpha\in \X^n\setminus \{0^n\}} \xi_\alpha 1_{Z(\alpha,e,0^n;Z(0^n))}  \,,$$
for some $\lambda_i,\xi_\alpha\in \ZZ$. Then since $f\in \ker \delta_1$, we have that 
\begin{align*}
\delta_1(f) & = \delta_1\left(\sum_{i=0}^k \lambda_i 1_{Z(0^n,g_i,0^n;Z(0^n))} + \sum_{\alpha\in \X^n\setminus \{0^n\}} \xi_\alpha 1_{Z(\alpha,e,0^n;Z(0^n))} \right)  \\ 
& =\delta_1\left(\sum_{\alpha\in \X^n\setminus \{0^n\}} \xi_\alpha 1_{Z(\alpha,e,0^n;Z(0^n))}\right)= \sum_{\alpha\in \X^n\setminus \{0^n\}} \xi_\alpha \left( 1_{Z(\alpha)}- 1_{Z(0^n)}\right)=0\,,   
\end{align*} 
but this forces $\xi_\alpha=0$ for every $\alpha\in \X^n\setminus\{0^n\}$. Thus,
$$[f]=\sum_{i=0}^k \lambda_i \left[ 1_{Z(0^n,g_i,0^n;Z(0^n))}\right] \,,$$
in $H_1((\HG_{(\Gamma,\X)})_n)$.

 In addition, using (\ref{eq_0}) again, we have the following relation
$$[1_{Z(0^n,g,0^n;Z(0^n))}]+[1_{Z(0^n,h,0^n;Z(0^n))}]=[1_{Z(0^n,gh,0^n;Z(0^n))}]$$
for every $g,h\in \Gamma$, so it follows that 
$$[1_{Z(0^n,gh,0^n;Z(0^n))}]=[1_{Z(0^n,hg,0^n;Z(0^n))}]\,,$$
for every $g,h\in\Gamma$. Recall that $\Gamma/[\Gamma,\Gamma]=\langle \bar{a},\bar{b}\rangle \cong \ZZ/2\ZZ\oplus\ZZ/2\ZZ$. Therefore we have that
$$[f]=\lambda_a \left[ 1_{Z(0^n,a,0^n;Z(0^n))}\right] + \lambda_b  \left[ 1_{Z(0^n,b,0^n;Z(0^n))}\right] \,,$$
where $\lambda_a,\lambda_b\in \{0,1\}$.

Now recall that the map 
$$(\iota_{n,n+1})_*:H_1((\HG_{(\Gamma,\X)})_n)\to H_1((\HG_{(\Gamma,\X)})_{n+1})$$
is given by 

\begin{align*}
(\iota_{n,n+1})_* \left( \left[ 1_{Z(0^n,g,0^n;Z(0^n))}\right] \right) & = (\iota_{n,n+1})_* \left( \left[ 1_{Z(0^n,g,0^n;Z(0^n0))}+1_{Z(0^n,g,0^n;Z(0^n1))} \right] \right) \\ 
& = \left[ 1_{Z(0^n(g\cdot 0),g_{|0},0^{n+1};Z(0^{n+1}))}+1_{Z(0^n(g\cdot 1),g_{|1},0^n1;Z(0^n1))} \right]  \\
& = \left[ 1_{Z(0^{n+1},g_{|0},0^{n+1};Z(0^{n+1}))}\right] + \left[ 1_{Z(0^{n+1},g_{|1},0^{n+1};Z(0^{n+1}))} \right]  \\
\end{align*}
for every $g\in \Gamma$.
Therefore,   we have that   
\begin{align*}
(\iota_{n,n+1})_*([1_{Z(0^n,a,0^n;Z(0^n))}]) & = [1_{Z(0^{n+1},e,0^{n+1};Z(0^{n+1}))}]+[1_{Z(0^{n+1},e,0^{n+1};Z(0^{n+1}))}]=0\,,
\end{align*}
in $H_1((\HG_{(\Gamma,\X)})_{n+1})$,
and 
\begin{align*}
(\iota_{n,n+1})_*([1_{Z(0^n,b,0^n;Z(0^n))}])& = [1_{Z(0^{n+1},a,0^{n+1};Z(0^{n+1}))}]+[1_{Z(0^{n+1},b,0^{n+1};Z(0^{n+1}))}]\,.
\end{align*}
Then, given $f\in C_c((\HG_{(\Gamma,\X)})_n, \ZZ)$ with $\delta_1(f)=0$ there exists big enough $m$ such that 
$$[f]=(\iota_{n,m})_*([f])=[\iota_{n,m}(f)]=\left[ \lambda 1_{Z(0^m,b,0^m;Z(0^m))}\right] \in H_1(\HG_{(\Gamma,\X)})\,,$$
with $\lambda\in \{0,1\}$.
Thus, the map $H_1(\HG_{(\Gamma,\X)})\to \ZZ/2\ZZ=\{\bar{0},\bar{1}\}$ given by 
$$[1_{Z(0^n,a,0^n;Z(0^n))}]\to \bar{0}\qquad \text{ and }\qquad [1_{Z(0^n,b,0^n;Z(0^n))}]\to \bar{1}$$
 is an isomorphism. 
\end{proof}

\begin{lemma}\label{higher_hom}
	Let $(\Gamma,X)$ be the self-similar infinite dihedral group, $\G_{(\Gamma,\X)}$ its associated groupoid of germs, and $\HG_{(\Gamma,\X)}:=c^{-1}(0)$. Then $H_{2n}(\HG_{(\Gamma,\X)})=0$ and $H_{2n+1}(\HG_{(\Gamma,\X)})=\ZZ/2\ZZ$ for every $n\geq 1$.
\end{lemma}
\begin{proof}
We denote by $\X^\infty\rtimes \Gamma$ the transformation groupoid of the action $\Gamma\curvearrowright \X^\infty$, and by $\RG_{2^n}$ the full equivalence relation of $2^n$ elements. Then one can check that $(\HG_{(\Gamma,\X)})_n$ is isomorphic to the product groupoid  $(\X^\infty \rtimes \Gamma)\times \RG_{2^n}$ via the map $[S_{(\alpha,g,\beta)}, x]\mapsto ((g,x),(\alpha,\beta))$. Then $(\HG_{(\Gamma,\X)})_n$ is Kakutani equivalent to $\X^\infty\rtimes \Gamma$, and so $H_*(\X^\infty\rtimes \Gamma)\cong H_*((\HG_{(\Gamma,\X)})_n)$ for every $n$. Then using for example \cite[Theorem 3.8(2)]{Mat2} we have that $H_*(\X^\infty\rtimes \Gamma)\cong H_*(\Gamma,C(\X^\infty,\ZZ))$ where $C(\X^\infty,\ZZ)$ is a left $\Gamma$-module by the induced action $\Gamma\curvearrowright C(\X^\infty,\ZZ)$. To simplify the calculation we are going to use that $\Gamma=\langle a\rangle \ast \langle b\rangle\cong \ZZ/2\ZZ\ast\ZZ/2\ZZ$, and that $$H_k(\langle a\rangle \ast \langle b\rangle,C(\X^\infty,\ZZ))\cong H_k(\langle a\rangle,C(\X^\infty,\ZZ))\oplus H_k(\langle b\rangle,C(\X^\infty,\ZZ))\,,$$
for every $k\geq 2$ \cite[Corollary 6.2.10]{Wei}. 

First observe that $H_*(\langle a\rangle,C(\X^\infty,\ZZ))\cong H_*(\X^\infty\rtimes\langle a\rangle)$, but the transformation groupoid $\X^\infty\rtimes\langle a\rangle$ is elementary, that is principal and compact, and so $H_k(\langle a\rangle,C(\X^\infty,\ZZ))=0$ for every $k\geq 1$. So it is enough to prove that $H_{2k}(\langle b\rangle,C(\X^\infty,\ZZ))=0$ and $H_{2k+1}(\langle b\rangle,C(\X^\infty,\ZZ))=\ZZ/2\ZZ$  for $k\geq 1$.

It is shown in \cite[Theorem 6.2.2]{Wei} that 
$$H_{2k}(\langle b\rangle,C(\X^\infty,\ZZ))\cong \ker(1+b)/(b-1)(C(\X^\infty,\ZZ))\,,$$
and 
$$H_{2k+1}(\langle b\rangle,C(\X^\infty,\ZZ))\cong \ker(b-1)/(b+1)(C(\X^\infty,\ZZ))\,,$$
for $k\geq 1$.
We start computing   $H_{2k}(\langle b\rangle,C(\X^\infty,\ZZ))$. First observe that given $N>0$ we have that $\X^\infty=\left( \bigsqcup_{n=0}^{N-1}\bigsqcup_{\alpha\in \X^*}^{N-1} (Z(1^n00\alpha)\sqcup Z(1^n01\alpha))\right ) \sqcup Z(1^N)$. Let $f\in C(\X^\infty,\ZZ))$, then we can write  $f=\sum_{n=0}^{N-1} (\sum_{\alpha\in \X^*} (\lambda_{1^n00\alpha} 1_{Z(1^n00\alpha)}+ \lambda_{1^n01\alpha} 1_{Z(1^n01\alpha)}))+\lambda_{1^{N}}1_{Z(1^{N})} $ for some $N>0$ and with only finitely many $\lambda$'s not equal to zero.

Now suppose that $f$ is in the kernel of $1+b$, then  we have that 
\begin{align*}
0 &=(1+b)f=(1+b)(\sum_{n=0}^{N-1} (\sum_{\alpha\in \X^*} (\lambda_{1^n00\alpha} 1_{Z(1^n00\alpha)}+ \lambda_{1^n01\alpha} 1_{Z(1^n01\alpha)}))+\lambda_{1^{N}}1_{Z(1^{N})}) \\
& =\sum_{n=0}^{N-1} (\sum_{\alpha\in \X^*} (\lambda_{1^n00\alpha} 1_{Z(1^n00\alpha)}+ \lambda_{1^n01\alpha} 1_{Z(1^n01\alpha)}))+\lambda_{1^{N}}1_{Z(1^{N})}+ \\ 
& +\sum_{n=0}^{N-1} (\sum_{\alpha\in \X^*} (\lambda_{1^n00\alpha} 1_{Z(b\cdot 1^n00\alpha)}+ \lambda_{1^n01\alpha} 1_{Z(b\cdot 1^n01\alpha)}))+\lambda_{1^{N}}1_{Z(b\cdot 1^{N})} \\
& =\sum_{n=0}^{N-1} (\sum_{\alpha\in \X^*} (\lambda_{1^n00\alpha} 1_{Z(1^n00\alpha)}+ \lambda_{1^n01\alpha} 1_{Z(1^n01\alpha)}))+\lambda_{1^{N}}1_{Z(1^{N})}+ \\ 
& +\sum_{n=0}^{N-1} (\sum_{\alpha\in \X^*} (\lambda_{1^n00\alpha} 1_{Z(1^n01\alpha)}+ \lambda_{1^n00\alpha} 1_{Z(1^n00\alpha)}))+\lambda_{1^{N}}1_{Z(1^{N})} \\
&= \sum_{n=0}^{N-1} (\sum_{\alpha\in \X^*} ((\lambda_{1^n00\alpha}+\lambda_{1^n01\alpha}) 1_{Z(1^n00\alpha)}+ (\lambda_{1^n01\alpha}+\lambda_{1^n00\alpha}) 1_{Z(1^n01\alpha)}))+2\lambda_{1^{N}}1_{Z(1^{N})}\,.
\end{align*}
But this forces $\lambda_{1^n01\alpha}+\lambda_{1^n00\alpha}=0$ and $\lambda_{1^N}=0$. Thus, we have that 
$$f=\sum_{n=0}^{N-1} (\sum_{\alpha\in \X^*} (\lambda_{1^n00\alpha} 1_{Z(1^n00\alpha)}- \lambda_{1^n00\alpha} 1_{Z(1^n01\alpha)}))=(1-b)(\sum_{n=0}^{N-1} \sum_{\alpha\in \X^*} \lambda_{1^n00\alpha} 1_{Z(1^n00\alpha)})\,.$$
So again we have that $\ker(1+b)=(1-b)(C(\X^\infty,\ZZ))$,  hence $H_{2k}(\langle b\rangle,C(\X^\infty,\ZZ))=0$.

Finally, we compute $H_{2k+1}(\langle b\rangle,C(\X^\infty,\ZZ))$.  Let $f\in \ker(b-1)$, so we can write  $f=\sum_{n=0}^{N-1} (\sum_{\alpha\in \X^*} (\lambda_{1^n00\alpha} 1_{Z(1^n00\alpha)}+ \lambda_{1^n01\alpha} 1_{Z(1^n01\alpha)}))+\lambda_{1^{N}}1_{Z(1^{N})} $ for some $N>0$ and with only finitely many $\lambda$'s not equal to zero. Then since $f\in \ker (b-1)$ and using similar computations as above we have that $\lambda_{1^n01\alpha}=\lambda_{1^n00\alpha}$, so 
$$f=\sum_{n=0}^{N-1} (\sum_{\alpha\in \X^*} (\lambda_{1^n00\alpha} 1_{Z(1^n00\alpha)}+ \lambda_{1^n00\alpha} 1_{Z(1^n01\alpha)}))+\lambda_{1^{N}}1_{Z(1^{N})}\,. $$ But 
$$(1+b)\left( \sum_{n=0}^{N-1} \sum_{\alpha\in \X^*} \lambda_{1^n00\alpha} 1_{Z(1^n00\alpha)} +\lambda_{1^{N}}1_{Z(1^{N})}\right)= \sum_{n=0}^{N-1} \sum_{\alpha\in \X^*} (\lambda_{1^n00\alpha} 1_{Z(1^n00\alpha)}+  \lambda_{1^n00\alpha} 1_{Z(1^n01\alpha)})+2\lambda_{1^{N}}1_{Z(1^{N})}\,.$$
Thus, it follows that $\ker(b-1)/(b+1)(C(\X^\infty,\ZZ))\cong \ZZ/2\ZZ$. 
\end{proof}

\begin{theorem}\label{groupoid_hom}

	Let $(\Gamma,X)$ be the self-similar infinite dihedral group, and let $\G_{(\Gamma,\X)}$ its associated groupoid of germs. Then $H_0(\G_{(\Gamma,\X)})=0$, and $H_1(\G_{(\Gamma,\X)})\cong H_2(\G_{(\Gamma,\X)})\cong \ZZ/2\ZZ$. Moreover, $H_n(\G_{(\Gamma,\X)})$ is a torsion group for every $n\geq 3$.

\end{theorem}

\begin{proof}
Let   $c:\G_{(\Gamma,\X)}\to \ZZ$ be the cocycle given by $c([S_{(\alpha,g,\beta)},  x])=|\alpha|-|\beta|$, and let $\G_{(\Gamma,\X)}\times_c \ZZ$ be the skew product groupoid defined in \cite{Mat2}. It is easy to check that $(\G_{(\Gamma,\X)}\times_c \ZZ)_{\G_{(\Gamma,\X)}^{(0)}\times\{0\}}\cong \HG_{(\Gamma,\X)}$ and that $\G_{(\Gamma,\X)}^{(0)}\times\{0\}$ is a full clopen subset of $\G_{(\Gamma,\X)}\times_c\ZZ$ (see for example \cite[Lemma 4.13]{Mat2}). Therefore $\HG_{(\Gamma,\X)}$ and $\G_{(\Gamma,\X)}\times_c \ZZ$ are Kakutani equivalent and hence $H_*(\HG_{(\Gamma,\X)})\cong H_*(\G_{(\Gamma,\X)}\times_c \ZZ)$ (induced by the natural inclusion).

 Let  $\sigma:\G_{(\Gamma,\X)}\times_c \ZZ\to \G_{(\Gamma,\X)}\times_c\ZZ$ be the shift map given by $g\times\{n\}\mapsto g\times\{n+1\}$. We are going to  describe the induced maps in homology $$\sigma_*:H_i(\HG_{(\Gamma,\X)})\cong H_i(\G_{(\Gamma,\X)}\times_c\ZZ)\to H_i(\G_{(\Gamma,\X)}\times_c\ZZ)\cong H_i(\HG_{(\Gamma,\X)})\,,$$
for $i=0,1$.


The map $\sigma_*:H_0(\G_{(\Gamma,\X)}\times_c\ZZ)\to H_0(\G_{(\Gamma,\X)}\times_c\ZZ)$ is given by $[1_{Z(0^n,e,0^n;Z(0^n))\times\{0\}}]\mapsto [1_{Z(0^n,e,0^n;Z(0^n))\times\{1\}}]$. 

Define the  bisection $U:=Z(0^n,e,0^{n+1};Z(0^{n+1}))\times\{1\}\subseteq\G_{(\Gamma,\X)}\times_c\ZZ$. Then 
$$U^{-1}\cdot (Z(0^n,e,0^n;Z(0^n))\times\{1\})\cdot U=Z(0^{n+1},e,0^{n+1};Z(0^{n+1}))\times\{0\}\,,$$
so $[1_{Z(0^n,e,0^n;Z(0^{n}))\times\{1\}}]=[1_{Z(0^{n+1},e,0^{n+1};Z(0^{n+1}))\times\{0\}}]$ in $H_0(\G_{(\Gamma,\X)}\times_c\ZZ)$, and so 
$$\sigma_*([1_{Z(0^n,e,0^n;Z(0^n))\times\{0\}}])=[1_{Z(0^{n+1},e,0^{n+1};Z(0^{n+1}))\times\{0\}}]\,,$$

  Therefore since $(\G_{(\Gamma,\X)}\times_c\ZZ)_{\G_{(\Gamma,\X)}^{(0)}\times\{0\}}\cong \HG_{(\Gamma,\X)}$ and following the isomorphism given in the proof of Proposition \ref{lower_hom}, the map $\sigma_*:H_0(\HG_{(\Gamma,\X)})\cong \ZZ[1/2]\to H_0(\HG_{(\Gamma,\X)})\cong\ZZ[1/2]$ is given by $x\to \frac{x}{2}$ for every $x\in\ZZ[1/2]$. 

Now, let $[1_{Z(0^n,g,0^n;Z(0^n))\times\{0\}}]\in H_1(\G_{(\Gamma,\X)}\times_c\ZZ)$, then applying the shift map we have that $$\sigma_*\left([1_{Z(0^n,g,0^n;Z(0^{n}))\times\{0\}}] \right)=[1_{Z(0^n,g,0^n;Z(0^n))\times\{1\}}]\in H_1(\G_{(\Gamma,\X)}\times_c\ZZ)\,.$$
Define again the bisection $U:=Z(0^n,e,0^{n+1};Z(0^{n+1}))\times\{1\}\subseteq\G_{(\Gamma,\X)}\times_c\ZZ$. Then 
$$U^{-1}\cdot (Z(0^n,g,0^n;Z(0^n))\times\{1\})\cdot U=Z(0^{n+1},g,0^{n+1};Z(0^{n+1}))\times\{0\}\,,$$
so 
$$[1_{Z(0^n,g,0^n;Z(0^n))\times\{1\}}]=[1_{Z(0^{n+1},g,0^{n+1};Z(0^{n+1}))\times\{0\}}]\in H_1(\G_{(\Gamma,\X)}\times_c\ZZ)\,.$$
Therefore, since $(\G_{(\Gamma,\X)}\times_c\ZZ)_{\G_{(\Gamma,\X)}^{(0)}\times\{0\}}\cong \HG_{(\Gamma,\X)}$, we have that 
$$\sigma_*\left([1_{Z(0^n,g,0^n;Z(0^n))}] \right)=[1_{Z(0^{n+1},g,0^{n+1};Z(0^{n+1}))}]\in H_1(\HG_{(\Gamma,\X)})\,.$$
But then again following the isomorphism given in the proof of Proposition \ref{lower_hom} we have that the map
$$\sigma_*:H_1(\HG_{(\Gamma,\X)})\cong \ZZ/2\ZZ\to \ZZ/2\ZZ\cong H_1(\HG_{(\Gamma,\X)})$$
is the identity map.
Finally, using the long exact sequence of homology  given in \cite[Lemma 1.3]{Ort}
\begin{equation}
\xymatrix {  0  &  H_0(\G_{(\Gamma,\X)})\ar@{>}[l] & H_0(\HG_{(\Gamma,\X)}) \ar@{>}[l] & H_0(\HG_{(\Gamma,\X)})\ar@{>}[l]_{\Ide-\sigma_*} & H_1(\G_{(\Gamma,\X)})\ar@{>}[l] & H_1(\HG_{(\Gamma,\X)}) \ar@{>}[l]  \\  \cdots \ar@{>}[r]& H_3(\G_{(\Gamma,\X)})\ar@{>}[r]& H_2(\HG_{(\Gamma,\X)}) \ar@{>}[r]^{\Ide-\sigma_*} &H_2(\HG_{(\Gamma,\X)}) \ar@{>}[r] & H_2(\G_{(\Gamma,\X)}) \ar@{>}[r]& H_1(\HG_{(\Gamma,\X)})\ar@{>}[u]_{\Ide-\sigma_*}   }\,,
\end{equation} 
we get

$$\xymatrix {  0  &  H_0(\G_{(\Gamma,\X)})\ar@{>}[l] & \ZZ[1/2] \ar@{>}[l] & \ZZ[1/2]\ar@{>}[l]_{\Ide-\sigma_*} & H_1(\G_{(\Gamma,\X)})\ar@{>}[l] & \ZZ/2\ZZ\ar@{>}[l]  \\  \cdots \ar@{>}[r]& H_3(\G_{(\Gamma,\X)})\ar@{>}[r]& 0 \ar@{>}[r]^{\Ide-\sigma_*} &0 \ar@{>}[r] & H_2(\G_{(\Gamma,\X)}) \ar@{>}[r]& \ZZ/2\ZZ\ar@{>}[u]_{\Ide-\sigma_*}   }\,,$$ 

It follows that $H_0(\G_{(\Gamma,\X)})=0$ and $H_1(\G_{(\Gamma,\X)})\cong H_2(\G_{(\Gamma,\X)})\cong \ZZ/2\ZZ$.

Moreover, by Lemma \ref{higher_hom} we have  the exact sequences
$$\xymatrix {  0  &  H_{2k+1}(\G_{(\Gamma,\X)})\ar@{>}[l] & H_{2k+1}(\HG_{(\Gamma,\X)}) \ar@{>}[l] & H_{2k+1}(\HG_{(\Gamma,\X)})\ar@{>}[l]_{\Ide-\sigma_*} & H_{2k+2}(\G_{(\Gamma,\X)})\ar@{>}[l] & 0 \ar@{>}[l]   }\,,
$$
for every $k\geq 1$.
Since $H_{2k+1}(\HG_{(\Gamma,\X)})$ are torsion groups so are $H_{2k+1}(\G_{(\Gamma,\X)})$ and $H_{2k+2}(\G_{(\Gamma,\X)})$.
\end{proof}

\subsection{Conclusion}

The $C^*$-algebra $C^*(\Gamma,\X)$ of the self-similar infinite dihedral group is isomorphic to the reduced groupoid $C^*$-algebra $C^*_r(\G_{(\Gamma,\X)})$ \cite[Corollary 5.7]{Nek}. On one hand, we have that $C^*(\Gamma,\X)$ is a unital, purely infinite simple $C^*$-algebra with $K$-theory $K_*(C^*(\Gamma,\X))=(\ZZ,\ZZ)$ (Proposition \ref{K-theory}). On the other hand $\G_{(\Gamma,\X)}$ is a minimal, effective, Hausdorff ample groupoid such that all its homology groups are torsion groups (Theorem \ref{groupoid_hom}). Therefore we have that 
$$\mathbb{Q}\cong K_i(C^*_r(\G_{(\Gamma,\X)}))\otimes \mathbb{Q}\ncong \sum_{k=0}^\infty H_{i+2k}(\G_{(\Gamma,\X)})\otimes \mathbb{Q}=0\qquad \text{for }i=0,1\,,$$
Therefore, $\G_{(\Gamma,\X)}$ contradicts the rational HK-conjecture.

\section{The AH-conjecture}
In this final section we will prove that the groupoid $\G_{(\Gamma,\X)}$ of the self-similar infinite dihedral group satisfies the AH-conjecture posted in \cite{Mat4}. This conjecture predicts that there is an exact sequence 
$$H_0(\G_{(\Gamma,\X)})\otimes \ZZ/2\ZZ\longrightarrow [[\G_{(\Gamma,\X)}]]_{ab}\longrightarrow H_1(\G_{(\Gamma,\X)})\longrightarrow 0\,.$$ 
This conjecture has been confirmed for several cases, and so far no counter-examples have been found. Among the cases the conjecture has been confirmed there are the Katsura-Exel-Pardo groupoids, or what is the same, the groupoids of a special self-similar action of $\ZZ$ over a finite graph \cite{NO}. 	With the techniques developed in \cite{NO}, that are inspired from the ones in \cite{Mat3}, one can see that an analysis of the properties of the kernel of the canonical cocycle is crucial. 

Most of the techniques used in \cite{NO} can be used in our case, since  we take the particular case of graph with one vertex and two loops, and replacing the group of integers by the infinite Dihedral group does not affect the results given in there.

First we will introduce notation and definitions. Let $\G$ be an  effective, Hausdorff ample groupoid. A bisection $U\subseteq \G$ is called \emph{full} if $r(U)=s(U)=\G^{(0)}$. Given a bisection $U$ with $r(U)\cap s(U)=\emptyset$ we define the full bisection 
$$\hat{U}:=U\sqcup U^{-1}\sqcup (\G^{(0)}\setminus \{r(U)\cup s(U)\})\,.$$

Then given a full bisection $U$ we define the homeomorphism $\pi_U: \G^{(0)}\to \G^{(0)}$ given by $s(\gamma)\mapsto r(\gamma)$ for every $\gamma\in U$. If $U$ is a bisection with $r(U)\cap s(U)=\emptyset$, we say that $\pi_{\hat{U}}$ is a \emph{transposition}.

Then we define the \emph{topological full group of $\G$} as the set $[[\G]]:=\{\pi_U:U\text{ full bisection}\}$ with composition as a product, so that $\pi_U\circ \pi_V=\pi_{UV}$. Then we can define the \emph{index map} $I:[[\G]]\to H_1(\G)$ as the homomorphism given by $\pi_U\mapsto [1_U]$. Since $H_1(\G)$ is an abelian group, we can induce the map $I_{ab}:[[\G]]_{ab}\to H_1(\G)$ where $[[\G]]_{ab}$ is the abelianization of $[[\G]]$.  We let $\TT(\G)$ denote the subgroup of $[[\G]]$ generated by all transpositions. One always has that $\TT(\G)\subseteq \ker(I)$, and having equality is closely related to the AH-conjecture.

\begin{definition}
	Let $\G$ be an effective, Hausdorff ample groupoid. We say that $\G$ has \emph{Property TR} if $\TT(\G)=\ker(I)$.
\end{definition}

It was shown that if the groupoid $\G$ is purely infinite (in the sense of \cite[Definition 4.9]{Mat3}) then Property TR is equivalent to the confirmation of the AH conjecture.

\begin{proposition}\label{TR_HG}
	Let $(\Gamma,\X)$ be the self-similar infinite dihedral group. Then the groupoid $\HG_{(\Gamma,\X)}$ satisfies Property TR.
\end{proposition}
\begin{proof}
	The proof is a verbatim of the proof of \cite[Lemma 5.3 and Proposition 5.4]{NO} with a slight modification because of the structure of the self-similar group $(\Gamma,\X)$.
	First observe that since $\HG_{(\Gamma,\X)}=\bigcup_{n=1}^\infty (\HG_{(\Gamma,\X)})_n$ and $(\HG_{(\Gamma,\X)})_n^{(0)}=\HG_{(\Gamma,\X)}^{(0)}$, we have that $[[(\HG_{(\Gamma,\X)})_n]]\subseteq [[(\HG_{(\Gamma,\X)})_{n+1}]]$ and $[[\HG_{(\Gamma,\X)}]]=\bigcup_{n=1}^\infty [[(\HG_{(\Gamma,\X)})_n]]$.

	Now suppose that $I(\pi_U)=0\in H_{1}(\HG_{(\Gamma,\X)})$, where  $\pi_U\in [[(\HG_{(\Gamma,\X)})_n]]$ for some $n\in \NN$.
	 Without loss of generality we can assume that $U=\bigsqcup_{i=1}^kZ(\alpha_i,g_i,\beta_i;Z(\beta_i))$ where $\alpha_i,\beta_i\in \X^n$ and $g_i\in \Gamma$ and $\bigsqcup_{i=1}^k Z(\alpha_i)=\bigsqcup_{i=1}^k Z(\beta_i)=\X^\infty$.
	 Then by the identifications given in the proof of Proposition \ref{lower_hom} we have that  
	 $$I(\pi_U)=\sum_{i=1}^k[1_{Z(0^n,\rho(g_i),0^n;Z(0^n))}]=[1_{Z(0^n,\rho(\prod_{i=1}^k g_i),0^n;Z(0^n)}]\,,$$
where $\rho:\Gamma\to \ZZ/2\ZZ$ is the group homomorphism that sends $a\mapsto \bar{0}$ and $b\mapsto \bar{1}$. Then since $I(\pi_U)=0\in H_1(\HG_{(\Gamma,\X)})$, it follows that $\rho(\prod_{i=1}^k g_i)=\sum_{i=1}^k\rho(g_i)=0$.	 But this means that $\prod_{i=1}^k g_i$ has the form either $a(bab)^l$ or $(bab)^la$ for some $l\geq 0$.

Now observe that $U=U_1\cdot U_2$ where 
$$U_1:=\bigsqcup_{i=1}^kZ(\alpha_i,g_i,\alpha_i;Z(\alpha_i))\qquad\text{and}\qquad U_2:=\bigsqcup_{i=1}^kZ(\alpha_i,e,\beta_i;Z(\beta_i))\,,$$
so $\pi_U=\pi_{U_1}\circ\pi_{U_2}$. It is shown in the proof of \cite[Lemma 5.3]{NO} that $\pi_{U_2}$ is a product of transpositions, so we can assume that 
$$U=U_1=\bigsqcup_{i=1}^kZ(\alpha_i,g_i,\alpha_i;Z(\alpha_i))$$
with $\rho(\prod_{i=1}^k g_i)=0$. Now given $1<i\leq k$ we define the bisections 
$$V_i:=Z(\alpha_1,g_i,\alpha_i;Z(\alpha_i))\qquad \text{and}\qquad W_i:=Z(\alpha_1,e,\alpha_i;Z(\alpha_i))\,.$$
then we have that
$$U\hat{V_2}\hat{W_2}\cdots \hat{V_k}\hat{W_k}=Z(\alpha_1,\prod_{i=1}^k g_i,\alpha_1;Z(\alpha_1))\sqcup \bigsqcup_{i=2}^k Z(\alpha_i,e,\alpha_i;Z(\alpha_i))\,.$$

First observe that if we define 
$$W_a:=Z(\alpha_11,a,\alpha_10 ;Z(\alpha_10))\qquad\text{and}\qquad W_{bab}:=Z(\alpha_11,bab,\alpha_10;Z(\alpha_10))\,,$$
then
$$Z(\alpha_1,a,\alpha_1;Z(\alpha_1))\sqcup \bigsqcup_{i=2}^k Z(\alpha_i,e,\alpha_i;Z(\alpha_i))=\hat{W}_a\,,$$
and 
$$Z(\alpha_1,bab,\alpha_1;Z(\alpha_1))\sqcup \bigsqcup_{i=2}^k Z(\alpha_i,e,\alpha_i;Z(\alpha_i))=\hat{W}_{bab}\,.$$
Then since $\rho(\prod_{i=1}^k g_i)=0$ we have that $\prod_{i=1}^k g_i$ is equal to $a(bab)^l$ or $(bab)^la$ for some $l\geq 0$. First suppose that $\prod_{i=1}^k g_i=a(bab)^l$, then 
$$U\hat{V_2}\hat{W_2}\cdots \hat{V_k}\hat{W_k}=\hat{W}_a \hat{W}_{bab}\cdots\hat{W}_{bab}\,.$$
Therefore, $\pi_U(\pi_{\hat{V_2}}\pi_{\hat{W_2}}\cdots \pi_{\hat{V_k}}\pi_{\hat{W_k}})=\pi_{\hat{W}_a}\pi_{\hat{W}_{bab}}^l$, and hence $\pi_U\in \TT(\HG_{(\Gamma,\X)})$, as desired. In a similar way we can prove it when  $\prod_{i=1}^k g_i=(bab)^la$.
	 
 \end{proof}

\begin{theorem}\label{AH-conjecture}
	Let $(\Gamma,\X)$ be the self-similar infinite dihedral group. Then the groupoid $\G_{(\Gamma,\X)}$ satisfies the AH-conjecture.
\end{theorem}
\begin{proof}
Since $\G_{(\Gamma,\X)}$ is a minimal purely infinite ample groupoid,  it is enough to prove that $\G_{(\Gamma,\X)}$ satisfies Property TR \cite[Theorem 4.4]{Mat4}. 
Let $I:[[\G_{(\Gamma,\X)}]]\to H_1(\G_{(\Gamma,\X)})$ and $I_\HG:[[\HG_{(\Gamma,\X)}]]\to H_1(\HG_{(\Gamma,\X)})$ be the respective index maps.
Let $\Phi:H_1(\HG_{(\Gamma,\X)})\to H_1(\G_{(\Gamma,\X)})$ be the map in the long exact sequence of the proof of Theorem \ref{groupoid_hom}, that is an isomorphism. An inspection of how to construct this map shows that $\Phi([1_{Z(0^n,b,0^n;Z(0^n))}])=[1_{Z(0^n,b,0^n;Z(0^n))}]$ and hence given a full bisection $U\subseteq \HG_{(\Gamma,\X)}$ we have that $I(\pi_U)=\Phi(I_\HG(\pi_U))$ (c.f. \cite[Lemma 4.6]{NO}). Then since $\Phi$ is an isomorphism, $I(\pi_U)=0$ if and only if $I_\HG(\pi_U)=0$. 

Now let  a full bisection $U\subseteq \G_{(\Gamma,\X)}$ such that $I(\pi_U)=0$. Then $U=\bigsqcup_{i=1}^k Z(\alpha_i,g_i,\beta_i;Z(\beta_i))$ for $\alpha_i,\beta_i\in \X^*$ and $g_i\in \Gamma$ such that $\X^\infty=\bigsqcup_{i=1}^k Z(\alpha_i)=\bigsqcup_{i=1}^k Z(\beta_i)$. Then if we define the full bisections $V=\bigsqcup_{i=1}^k Z(\alpha_i,g_i,\alpha_i;Z(\alpha_i))\subseteq \HG_{(\Gamma,\X)}$ and $W=\bigsqcup_{i=1}^kZ(\alpha_i,e,\beta_i;Z(\beta_i))$, we have that $U=VW$, and hence 
$$0=I(\pi_U)=I(\pi_V\pi_W)=I(\pi_V)+I(\pi_W)\,.$$
Now observe that $\pi_W$ can be seen as an element of the topological full group of the open subgroupoid 
$$\G_2:=\{[\alpha,e,\beta;x]\in \G_{(\Gamma,\X)}: \alpha,\beta\in \X^*,\,x\in Z(\beta) \}\,,$$ 
that is isomorphic to the groupoid of one-sided full shift with two letters, so $[[\G_2]]\subseteq [[\G_{(\Gamma,\X)}]]$. It is known that $\pi_W\in\TT(\G_2)\subseteq\TT(\G_{(\Gamma,\X)})$ by \cite[Lemma 6.10]{Mat3}, so $I(\pi_W)=0$. Therefore 
$I(\pi_U)=I(\pi_V)=0$, and then $I_\HG(\pi_U)=0$. Thus, by Proposition \ref{TR_HG} it follows that $\pi_U\in \TT(\HG_{(\Gamma,\X)})\subseteq\TT(\G_{(\Gamma,\X)})$, and hence $\pi_U=\pi_V\pi_W\in \TT(\G_{(\Gamma,\X)})$ as desired.  
\end{proof}

\begin{corollary}
	Let $(\Gamma,\X)$ be the self-similar infinite dihedral group. Then $[[\G_{(\Gamma,\X)}]]_{ab}\cong \ZZ/2\ZZ$ and it is generated by the class of $\pi_U$, where $U=Z(\emptyset,b,\emptyset;\X^\infty)$. 
\end{corollary}
\begin{proof}
	By Theorem \ref{AH-conjecture} and Theorem \ref{groupoid_hom} we have that $[[\G_{(\Gamma,\X)}]]_{ab}\cong H_1(\G_{(\Gamma,\X)})\cong \ZZ/2\ZZ$. Finally the last statement follows from an inspection of the proofs of Proposition \ref{TR_HG} and Theorem \ref{AH-conjecture}.
\end{proof}

\end{document}